\documentclass[11pt,a4paper]{amsart}

\usepackage[left=1in,right=1in,top=1in,bottom=1in,foot=0.5in]{geometry}

\usepackage{amsmath,amsfonts,amssymb}
\usepackage[T1]{fontenc}
\usepackage{color}
\usepackage{epsfig}
\usepackage{graphicx}
\usepackage{cite,url}
\usepackage{hyperref}
\usepackage{xspace}
\usepackage{enumerate}

\newtheorem{mytheorem}{Theorem}[section]
\newtheorem{mylemma}[mytheorem]{Lemma}
\newtheorem{myproposition}[mytheorem]{Proposition}
\newtheorem{mycorollary}[mytheorem]{Corollary}

\newtheorem{myconjecture}[mytheorem]{Conjecture}

\theoremstyle{definition}

\newtheorem{myremark}[mytheorem]{Remark}

\newcommand{\commentout}[1]{}

\begin{document}

\title{Towards a more structured search for Erd\H{o}s-Gy\'arf\'as counter-examples}
\author[G.\ Ducoffe]{Guillaume Ducoffe} \address{National Institute
 for Research and Development in Informatics and University of
  Bucharest, Bucureşti, Rom\^{a}nia} \email{guillaume.ducoffe@ici.ro}
\author[B.\ Dumitru]{Bogdan Dumitru} \address{University of
  Bucharest, Bucureşti, Rom\^{a}nia} \email{bogdan.dumitru@fmi.unibuc.ro}
\date{}

\maketitle

\begin{abstract}
    The Erd\H{o}s-Gy\'arf\'as conjecture posits that every graph with minimum degree at least three contains a cycle of length some power of two.
    We prove a few simple structural properties for any minimal counter-example to this conjecture.
    In particular, the fraction of its vertices of degree three must be greater than $2/3$, thus improving on the prior bound of $4/7$ (Carr, 2026). 
    Furthermore, it is either biconnected or the $1$-clique-sum of two biconnected graphs.
    By exploiting some of these properties, we were able to verify the conjecture for every graph of order at most $40$, every bipartite graph of order at most $66$, and every cubic graph of order at most $48$.
\end{abstract}

\section{Introduction}

Graphs in what follows are undirected, simple (\textit{i.e.}, loopless and without multiple edges), and unweighted.
We refer to~\cite{BoMu08} for a standard textbook on Graph Theory.
Erd\H{o}s and Gy\'arf\'as conjectured that every graph with minimum degree at least three contains a cycle of length some power of two~\cite{Erd97}.
It is remarkable that such a stringent hypothesis has remained open after 30 years.
Note that there is an abundance of $C_4$'s in random graphs $G(n,p)$, with $p$ a constant. Therefore, the conjecture is true for almost all graphs. 
It is also true for some classes of graphs with strong metric properties.
For instance, consider a graph $G$ such that $G \setminus v$ is distance-preserving for some $v$.
{Assume that $G$ has minimum degree at least three.
If $v$ has two nonadjacent neighbors $x$ and $y$, then the distance preservation in $G \setminus v$ gives another common neighbor $u \ne v$, and $v\,x\,u\,y\,v$ is a $C_4$.
Otherwise, $N_G(v)$ is a clique of size at least three, and together with $v$ it contains a $K_4$, hence a $C_4$.}
Consequently, the conjecture is satisfied by every graph with a distance-preserving elimination ordering, including: chordal graphs~\cite{Ros70}, distance-hereditary graphs~\cite{BaMu86}, cop-win graphs, and weakly modular graphs -- a broad class of graphs studied in Metric Graph Theory, see~\cite{CCHO20}.
Recently, Carr proved the conjecture for the graphs with diameter two~\cite{Car25}.
The conjecture also holds for $3$-connected cubic planar graphs~\cite{HeKr13}.
To the best of our knowledge, the latter result has been the only one to date to use connectivity properties in the study of the Erd\H{o}s-Gy\'arf\'as conjecture.
We refer to~\cite{DaSh01,Sha98} for other partial results.

In~\cite{Royle}, Royle verified the conjecture for every graph of order at most $15$.
In~\cite{Mark04}, Markstr\"om further verified the conjecture for every cubic graph of order at most $29$.
Recently, these bounds have been improved in several concurrent works, with the best known results at the time of writing these lines being: $31$ for general graphs~\cite{Bal26} {(hence also for cubic graphs)}, and $58$ for bipartite cubic graphs~\cite{Tra26}. The related problem of computing the smallest order $f(k)$ of a cubic graph with no cycle of length $4,8,16,\ldots,2^k$ has been considered in~\cite{Gar26}. 
Hereafter, for any class of graphs $\mathcal{F}$, a \emph{counter-example} within $\mathcal{F}$ refers to any graph $G \in \mathcal{F}$ such that $G$ has minimum degree at least three but no cycle of length some power of two.
It is called a \emph{minimal} counter-example within $\mathcal{F}$ if it minimizes $(|V(G)|,|E(G)|)$ with respect to the lexicographic order.
When $\mathcal{F}$ is the class of all graphs, then, we simply refer to a minimal counter-example.
Improving the aforementioned bounds is likely to require additional insights on the structure of a minimal counter-example.
In this respect, it was proved in~\cite{Car26} that any minimal counter-example of order $n$ must contain at least $\frac{4n}7$ vertices of degree three.
Tranquilli related the search for bipartite cubic counter-examples to combinatorial design theory~\cite{Tra26}.

\medskip
\noindent
\textbf{Our contributions.}
We prove several new properties for any minimal counter-example to the Erd\H{o}s-Gy\'arf\'as conjecture.
In particular, we prove that a minimal counter-example is always $2$-edge connected, and it is almost biconnected.
These connectivity results are first reported in Sec.~\ref{sec:connectivity}.
Furthermore, any minimal counter-example of order $n$ has at most $2n-2$ edges, and it must contain at least $\frac{2n}3+1$ vertices of degree three.
The latter improves on~\cite{Car26}.
These results are reported in Sec.~\ref{sec:eg-density}. 
Then, based on some of the results above, we propose two reformulations of the conjecture in Sec.~\ref{sec:conjecture}.
Finally, we verify the conjecture for every graph of order at most $40$, every bipartite graph of order at most $66$, and every cubic graph of order at most $48$ (Sec.~\ref{sec:verif}). 

\smallskip
We conclude this paper in Sec.~\ref{sec:perspectives} with some perspectives for future works.

\medskip
\noindent
\textbf{Notations.}
Let $G=(V,E)$ be a graph.
Throughout the paper, we denote by $n_G := |V|$ its order, and by $m_G := |E|$ its size.
Two vertices $u$ and $v$ are called adjacent if and only $uv \in E$.
The neighborhood of a vertex $v$ is made of all the vertices $u$ adjacent to $v$, and it is denoted by $N_G(v)$.
Let also $d_G(v) := |N_G(v)|$ be denoting the degree of $v$.
For every $X \subseteq V$, the induced subgraph $G[X]$ has for vertex set $X$ and for edge set $E_G(X) := \{ uv \in E : u,v \in X \}$.
In the same way, for every $E' \subseteq E$, the subgraph induced by $E'$ has for edge set $E'$ and for vertex set the end-vertices of all the edges in $E'$.
We omit $G$ from the subscript of our notations if it is clear from the context.
Other notations and terminology are locally defined wherever they are needed in the paper.

\section{Connectivity}\label{sec:connectivity}

We present new results on the connectivity of any minimal counter-example.
Our proofs in what follows rely on edge- and vertex-deletions, edge contractions, and $1$-clique-sum -- \textit{a.k.a.}, the gluing of two graphs along a common vertex.
Therefore, our results in this section can be also applied to any minimal counter-example within a class of graphs $\mathcal{F}$ that is stable under these operations (\textit{e.g.}, planar graphs). 

\begin{mylemma}\label{lm:cycle-edge}
  In a minimal counter-example $G=(V,E)$, every edge is contained in a cycle of length $2^k+1$, for some integer $k \ge 1$.  
\end{mylemma}
\begin{proof}
    Let $e \in E$ be arbitrary. 
    If $e$ belongs to some triangle, then the result follows for $k=1$.
    Otherwise, let $G' = G / e$ be the graph that results from the contraction of the edge $e$ (removing loops and multiple edges).
    Since $e$ does not belong to a triangle, its two end-vertices do not have common neighbors in $G$.
    Therefore, the minimum degree of $G'$ is at least three.
    By minimality of $G$, $G'$ contains a cycle of length $2^k$, for some integer $k \ge 2$. 
    As this cycle does not exist in $G$, it must contain the vertex $v_e$ resulting from the contraction of the edge $e$. 
    Therefore, there is a cycle of length $2^k+1$ in $G$ that contains the edge $e$.
\end{proof}

\begin{mycorollary}\label{cor:2-edge-connected}
    Any minimal counter-example is $2$-edge connected.
\end{mycorollary}
\begin{proof}
    Since by Lemma~\ref{lm:cycle-edge} every edge is in a cycle, there can be no bridge.    
\end{proof}

The following result is our main result in this section:

\begin{mytheorem}\label{thm:2-vertex-connected}
    If $G$ is a minimal counter-example, and it is not biconnected, then the following must hold:
    \begin{enumerate}
        \item There is a unique cut-vertex $c$;
        \item There are exactly two connected components $A,B$ of $G - c$;
        \item Vertex $c$ has exactly two neighbors $a_1,a_2$ in $A$, each of degree three;
        \item Vertex $c$ has exactly two neighbours $b_1,b_2$ in $B$, each of degree three;
        \item $|A| = |B|$;
        \item $|E(A)| = |E(B)|$;
        \item Both the blocks $G[A \cup \{c\}],G[B \cup \{c\}]$ are $2$-degenerate.
    \end{enumerate}
\end{mytheorem}
\begin{proof}
    Assume the existence of a cut-vertex $c$ (if none exists, then $G$ is biconnected, and so we are done). Consider the connected components $C_1,C_2,\ldots,C_q$ of $G - c$. By minimality of $G$, the removal of any component $C_j$ must result in vertex $c$ having degree strictly less than $3$. Therefore, $q \le 3$. Furthermore, if $q = 3$, then every edge incident to $c$ is a bridge, thus contradicting Corollary~\ref{cor:2-edge-connected}. As a result, $q = 2$ holds. Furthermore, $c$ has at least two neighbors in each component. If $c$ has at least three neighbors in say $C_1$, then the counter-example $G[C_1 \cup \{c\}]$ would contradict the minimality of $G$. Hence, $c$ has exactly two neighbors in each component. In what follows, let $A = C_1, B = C_2$. Then, $N(c) = \{a_1,a_2,b_1,b_2\}$ for some $a_1,a_2 \in A$ and $b_1,b_2 \in B$. Since $c$ has degree four, the minimality of $G$ also implies that all its neighbors have degree three (otherwise, we could have removed any edge between $c$ and one of its neighbors of degree at least four). Suppose now by contradiction $|A| < |B|$. By replacing $B$ with a disjoint copy $A'$ of $A$, connecting $c$ to $a_1',a_2'$, we would get a smaller counter-example than $G$. As a result, we obtain $|A| = |B|$. By a similar reasoning, we also obtain $|E(A)| = |E(B)|$. Note that $G[A \cup \{c\}]$ does not contain any cycle whose length equals a power of two. Therefore, if $G[A \cup \{c\}]$ would contain a subgraph $H_A$ of minimum degree at least three, then, the latter would contradict the minimality of $G$. It implies that $G[A \cup \{c\}]$ and, by similar reasoning, $G[B \cup \{c\}]$, are $2$-degenerate. Finally, we argue that there is no other cut-vertex $c' \ne c$. Indeed, if such a $c'$ would exist, then, say it is in $A$. By the same reasoning as above, there are two equal-size connected components $A',B'$ in $G - c'$. However, say that $c \in B'$. Then, $B \subseteq B'$. It implies that $|A'| \le |A \setminus \{c'\}| = |A|-1 < |B| < |B'|$. A contradiction. 
\end{proof}

\begin{mycorollary}\label{cor:even-biconnected}
    If a minimal counter-example has an even number of vertices, then it is biconnected.   
\end{mycorollary}
\begin{proof}
    If $G$ is a minimal counter-example and it is not biconnected, then, by Theorem~\ref{thm:2-vertex-connected}, it has a unique cut-vertex $c$ such that $G - c$ is made of two equal-size connected components $A$ and $B$. In particular, the order of $G$ is $n = |A|+|B|+1 = 2|A|+1$, which is odd. 
\end{proof}

Note that bipartite graphs and cubic graphs are not stable under edge contractions.
Therefore, our results above cannot be applied to minimal bipartite counter-examples, nor to minimal cubic counter-examples.
However, a slightly weaker result can be derived for cubic graphs, and also for \emph{bridge-addable} classes of graphs that are both hereditary {and closed under disjoint unions}.
Recall that a class of graphs is called bridge-addable if it is stable under adding an edge between two distinct connected components, see~\cite{AMR12}. In particular, bipartite graphs are {hereditary, bridge-addable, and closed under disjoint unions}.

\begin{mytheorem}\label{thm:cubic-connectivity}
    Let $G$ be any minimal cubic counter-example, or any minimal counter-example within a class of graphs $\mathcal{F}$ that is {hereditary, bridge-addable, and closed under disjoint unions}.
    Either $G$ is $2$-edge connected, or the following must hold:
    \begin{enumerate}
        \item There is a unique bridge $e$;
        \item The two connected components $A,B$ of $G \setminus e$ satisfy $|A| = |B|$, and $|E(A)| = |E(B)|$;
        \item Both subgraphs $G[A]$ and $G[B]$ are $2$-degenerate.
    \end{enumerate}
\end{mytheorem}

The proof is very similar to that of Theorem~\ref{thm:2-vertex-connected}, and therefore it is omitted.

\section{The number of vertices of degree three}\label{sec:eg-density}

This section is devoted to bounds on the number of degree-three vertices in any minimal counter-example, and to various related results.
Some general bounds are first presented in Sec.~\ref{ssec:setup}.
Afterwards, our main techniques in this part, which we call the \emph{doubling reduction}, is introduced in Sec.~\ref{ssec:doubling-reduction}.
Roughly, the latter is a contraction scheme that halves cycle lengths (see Lemma~\ref{lem:D} in what follows).
We note that Tranquilli has implicitly used a different cycle-halving technique than ours, as he considers the two half-squares of any cubic bipartite counter-example~\cite{Tra26}.
By contrast, our technique applies to any minimal counter-example, which is not necessarily cubic nor bipartite.
Further refinements of this technique are discussed in Sec.~\ref{ssec:subdivisions}, where we show that it applies at every dyadic scale (Theorem~\ref{thm:Dplus}).

\medskip
The following folklore bound is used in what follows:
\begin{mylemma}\label{lem:2-degenerate-bound}
    If $G$ is a $2$-degenerate graph with at least $\kappa$ connected components and at most $\iota$ isolated vertices, then, $m \le 2n - 3\kappa + \iota$ holds.
\end{mylemma}
\begin{proof}
    Since the upper bound decreases with $\kappa$ but increases with $\iota$, we may assume in what follows there are exactly $\kappa$ connected components and $\iota$ isolated vertices.
    Consider any connected component $C_j$ with at least two vertices.
    Since $G[C_j]$ is $2$-degenerate, we can iteratively remove vertices of degree at most two from this component until there has remained two vertices of $C_j$.
    At each iteration, at most two edges of $E(C_j)$ are removed. Furthermore, there is at most one edge between the two remaining vertices of $C_j$.
    Therefore, $|E(C_j)| \le 2(|C_j|-2)+1 = 2|C_j|-3$.
    By summing over all the $\kappa-\iota$ connected components with at least two vertices, we get that $m \le 2(n-\iota)-3(\kappa-\iota) = 2n - 3\kappa + \iota$.
\end{proof}

    \subsection{General bounds}\label{ssec:setup}
    Results in this part can be applied to any minimal counter-example within a monotone {class of graphs (\textit{i.e.}, closed under taking subgraphs)}.
    In particular, they hold for bipartite graphs.
    Indeed, we only use vertex- and edge-deletions in our proofs.

\medskip
    For any graph $G=(V,E)$, let us write
    \[
        L=\{v: d(v)=3\},\qquad H=\{v: d(v)\geq 4\},\qquad
        \ell=|L|,\quad h=|H|.
    \]
    We use the following standard Royle-like reductions.

    \begin{mylemma}\label{lem:red}
        The following hold for any minimal counter-example $G$:
        \begin{enumerate}
            \item no edge joins two vertices of $H$;
            \item every vertex of $L$ has a neighbor in $L$;
            \item every neighbor of a vertex of $H$ lies in $L$;
            \item every vertex of $H$ has at least four neighbors in $L$;
            \item $G$ contains no proper subgraph of minimum degree at least three.
        \end{enumerate}
    \end{mylemma}
   \begin{proof}
    First observe that no cycle can be created by deleting an edge.
    Therefore, by minimality of $G$, every edge must have an end-vertex in $L$.
    The latter implies (1) and (3). Since every vertex of $H$ has degree at least four by definition, it also implies (4).
    In the same way, observe that no cycle can be created either by deleting a vertex.
    The latter implies (2).
    Finally, as any proper subgraph of $G$ has no power-of-two cycle, its minimum degree must be at most two for otherwise the latter would contradict the minimality of $G$.
    The latter implies (5).
    \end{proof}

    We first use Lemma~\ref{lem:red} to upper bound the number of edges in any minimal counter-example.
    \begin{mytheorem}\label{thm:edge-bound}
        For any minimal counter-example $G=(V,E)$, $m \le 2n-2$ holds.
    \end{mytheorem}
    \begin{proof}
        By Lemma~\ref{lem:red}(1), $G$ has a degree-three vertex $v$.
        Furthermore, by Lemma~\ref{lem:red}(5), $G \setminus v$ is $2$-degenerate.
        Therefore, by Lemma~\ref{lem:2-degenerate-bound} applied for $G \setminus v$ with $\kappa = 1, \ \iota = 0$, we obtain $m - d(v) \le 2(n-1) - 3$.
        Since $d(v) = 3$, the latter simplifies to $m \le 2n-2$.
    \end{proof}

    Next, we give a first lower bound on $\ell = |L|$, which improves on the state of the art~\cite{Car26}.
    \begin{myproposition}\label{prop:density}
        For any minimal counter-example $G=(V,E)$, $\ell \ge 2n/3$ holds.
    \end{myproposition}
    \begin{proof}
       By Lemma~\ref{lem:red}(2), each $v\in L$ has at most two neighbors in $H$.
        The latter suggests to partition $L$ as follows:
        \[
            L_i=\{v\in L:\ v\ \text{has exactly}\ i\ \text{neighbors in}\ H\},\qquad
            i=0,1,2 .
        \]
        Let us write $e(L,H)$ for the number of edges with one end in $L$ and one end in $H$.
        Note that by  Lemma~\ref{lem:red}(3), the number of edges incident to a vertex of $H$ is exactly $e(L,H)$. 
        Therefore, by using a double counting argument we obtain
        \begin{equation}\label{eq:count}
            4h\ \leq\ e(L,H)\ =\ |L_1|+2|L_2| .
        \end{equation}
        Together with $|L_1|+2|L_2|\le 2\ell$ we get $\ell\geq 2h$, and so,\ $\ell\geq 2n/3$.     
    \end{proof}

    \subsection{The doubling reduction}\label{ssec:doubling-reduction}

    The purpose of this part is to introduce the doubling technique, which we use in order to sharpen our lower bound on the number of degree-three vertices in any minimal counter-example.
    Since we use vertex- and edge-deletions, but also edge contractions, the results in what follows can be applied to any minimal counter-example within a minor-closed class of graphs.
    Note that in what follows we are re-using the partition $L_0,L_1,L_2$ of $L$ from Prop.~\ref{prop:density}.
    \begin{mylemma}\label{lem:D}
    For any minimal counter-example $G=(V,E)$, let us define a multigraph $M$ with vertex set $H$ such that, for each vertices $x$ and $y$ of $H$, the number of edges $xy$ in $M$ is equal to the number of vertices $v\in L_2$ such that $N_G(v) \cap H = \{x,y\}$.
    Then:
        \begin{enumerate}
            \item $M$ is simple;
            \item $M$ contains no cycle of length $2^k$, for any $k\geq 1$;
            \item $M$ is $2$-degenerate.
        \end{enumerate}
        In particular, $|L_2|=|E(M)|\leq 2h-3$ when $h\geq 2$, and $L_2=\emptyset$ when $h\leq 1$.
    \end{mylemma}
    \begin{proof}
        Suppose by contradiction there exist multiple edges $xy$ in $M$ for some $x,y \in H$.
        Then, there exist distinct vertices $u,v \in L_2$ such that $N(u) \cap H = N(v) \cap H = \{x,y\}$.
        However, it implies the existence of a cycle $x\,u\,y\,v\,x$, thus contradicting our assumption that $G$ is a counter-example.
        The latter implies (1).

        As $M$ is simple, it has no cycle of length $2^k$ for $k=1$.
        Furthermore, let $x_0x_1\cdots x_{t-1}$ be a simple cycle in $M$, of length $t \ge 3$.
        Again using the fact that $M$ is simple, there exist pairwise different vertices $u_0,u_1,\ldots u_{t-1}$ such that, for every $i$ with $0 \le i < t$, $N(u_i) \cap H = \{x_i,x_{i+1}\}$ holds (indices are taken modulo $t$). Then,
        \[
        x_0\,u_0\,x_1\,u_1\,\cdots\,x_{t-1}\,u_{t-1}\,x_0
        \]
        is a simple cycle of length $2t$ in $G$. Therefore, if $t=2^k$ for some $k\geq 2$, this would contradict our assumption that $G$ is a counter-example.
        The latter implies (2).

        Suppose now by contradiction there exists a subgraph $M'$ of $M$ with minimum degree at least three (possibly, $M = M'$).
        As $M$ is simple, so is $M'$, and as $M$ has no power-of-two cycle, neither does $M'$.
        Hence, $M'$ is a counterexample.
        However, by Lemma~\ref{lem:red}(1) we obtain that $\ell \ge 1$, and so, $|M'| \le |M| = h = n - \ell < n$, thus contradicting the minimality of $G$.
        The latter implies (3).

        Finally, as $M$ is simple, its number of edges equals $|L_2|$.
        If $h \le 1$, then clearly $L_2 = \emptyset$.
        From now on, $h \ge 2$.
        Since $2h - 3 \ge 1$, the inequality $|L_2| \le 2h-3$ vacuously holds if $L_2 = \emptyset$.
        Otherwise, $\kappa \ge \iota + 1$, with $\kappa$ and $\iota$ respectively denoting the number of connected components and the number of isolated vertices of $M$.
        As $M$ is $2$-degenerate, by Lemma~\ref{lem:2-degenerate-bound} we obtain $|L_2| \le 2h - 3\kappa + \iota \le 2h - 3 - 2\iota \le 2h-3$.
        \end{proof}

        \begin{mycorollary}\label{cor:notall}
            $L=L_2$ is impossible in a minimal counterexample.
        \end{mycorollary}
        \begin{proof}
            If $L=L_2$ then by Lemma~\ref{lem:red}(3), for every $x\in H$, all of its neighbors lie in $L_2$.
            It implies $d_M(x)=d_G(x)\geq 4$ for all $x\in H$, where $M$ is the multigraph defined in Lemma~\ref{lem:D}.
            In particular, $M$ has minimum degree at least four, thus contradicting Lemma~\ref{lem:D}(3).
        \end{proof}

        We can finally state the main result of this part:
        \begin{mytheorem}\label{thm:density}
            For any minimal counter-example $G=(V,E)$, $\ell \ge 2n/3 + 1$ holds.
        \end{mytheorem}
        \begin{proof}
            It suffices to prove $h \le n/3-1$.
            This is true if $h=0$ because $n \ge 4$, and so, $n/3 -1 > 0$.
            Assume now $h = 1$.
            Since the unique vertex of $H$ has at least four neighbors, $n \ge 5$ holds.
            However, if $n = 5$, then the four vertices of $L$ must induce a $2$-regular graph, and so, a $C_4$, thus contradicting our assumption that $G$ is a counter-example.
            Therefore, $n \ge 6$, and so, $n/3-1 \ge 1$ holds.
            Finally, for $h\geq 2$, we obtain from \eqref{eq:count} and Lemma~\ref{lem:D} that
            \[
                \ell\ \geq\ |L_1|+|L_2|\ =\ \bigl(|L_1|+2|L_2|\bigr)-|L_2|
                \ \geq\ 4h-(2h-3)\ =\ 2h+3 .
            \]
            Then, $n=\ell+h\geq 3h+3$, or equivalently, $h \le n/3 -1$.
        \end{proof}

    \subsection{Iterating the reduction}\label{ssec:subdivisions}
    Recall for what follows that a \emph{$t$-subdivision} of a multigraph $K$ is the graph obtained by replacing all the edges of $K$ with internally vertex-disjoint paths of length $t$.
    -- So, in particular, a loop is replaced by a cycle of length $t$. --
    We now generalize Lemma~\ref{lem:D}, as follows:

    \begin{mytheorem}\label{thm:Dplus}
    If $G$ is any minimal counterexample, then, it does not contain any $2^j$-subdivision of $K$, for any integer $j \ge 1$ and any multigraph $K$ with minimum degree at least three.
    \end{mytheorem}

\begin{proof}
Suppose by contradiction the existence of some $j$ and $K$ such that $G$ contains a $2^j$-subdivision of $K$.
In what follows, we denote by $S$ the subdivision.
For any integer $t \ge 0$, the existence of a cycle of length $2^t$ in $K$ would result in a cycle of length $2^{j+t}$ in $S$, which is forbidden in $G$.
In particular, $K$ has no loops ($t=0$), and no parallel edges ($t=1$).
It implies that $K$ is a simple graph of minimum degree at least three with no power-of-two cycles, \textit{i.e.}, a counter-example.
However, because we assume $j \ge 1$, $|V(K)| < |V(S)| \le n$ should hold, thus contradicting the minimality of $G$.
\end{proof}

We note that the proof of Lemma~\ref{lem:D} essentially reduces to the case $j=1$.

\section{Some new conjectures}\label{sec:conjecture}

We introduce two conjectures in what follows, which we relate to the Erd\H{o}s-Gy\'arf\'as conjecture.

    \subsection{Allowing one vertex of degree at most two}
    Our first conjecture is a modest loosening of the original Erd\H{o}s-Gy\'arf\'as conjecture.

        \begin{myconjecture}\label{conj:deg-2}
            Every graph with at most one vertex of degree at most two and at least one vertex of degree at least three has a power-of-two cycle.
        \end{myconjecture}

    Based on our techniques in Sec.~\ref{sec:connectivity}, the following equivalence is proved:

        \begin{mytheorem}
        Conjecture~\ref{conj:deg-2} is equivalent to the Erd\H{o}s-Gy\'arf\'as conjecture.
        \end{mytheorem}
        \begin{proof}
            If Conjecture~\ref{conj:deg-2} is true, then in particular it holds for graphs of minimum degree at least three, thus implying the Erd\H{o}s-Gy\'arf\'as conjecture.
            Conversely, assume the existence of some counter-example $G$ to Conjecture~\ref{conj:deg-2}.
            If $G$ has minimum degree at least three, then, it is also a counter-example to the Erd\H{o}s-Gy\'arf\'as conjecture.
            Thus from now on, we assume the existence of a unique vertex $v$ of $G$ such that $d(v) \le 2$.
            If $d(v) = 0$, then, $G \setminus v$ is a counter-example to the Erd\H{o}s-Gy\'arf\'as conjecture.
            Furthermore, if $d(v)=1$, then, $G \setminus v$ is a smaller counter-example to Conjecture~\ref{conj:deg-2} with at most one vertex of degree two (namely, the former neighbor of $v$) and all the other vertices of degree at least $3$.
            Hence, we may assume in what follows $d(v) = 2$.
            In this situation, we take two copies $G_1,G_2$ of $G$, and we identify their degree-two vertices $v_1,v_2$, which results in a counter-example to the Erd\H{o}s-Gy\'arf\'as conjecture.
        \end{proof}

    It could be interesting to study the natural variant of Conjecture~\ref{conj:deg-2} where we allow up to $k$ vertices of degree at most two, for some constant $k \ge 2$.

    \subsection{A detour through edge-colored graphs}

    By an \emph{edge-colored graph}, we mean a pair $(G,\texttt{col})$ where $G=(V,E)$ is a graph and $\texttt{col} : E \mapsto \mathbb{N}$.
    A \emph{rainbow cycle} is a cycle of $G$ whose all edges have pairwise different colors.

        \begin{myconjecture}\label{conj:edge-colors}
            Let $(G,\texttt{col})$ be an edge-colored graph such that every vertex is incident to at least three colors, and the edges of any color induce a clique of size at least three.
            There is a rainbow cycle in $G$ of length equal to some power of two.
        \end{myconjecture}

        Roughly, we increase the bound on the minimum degree, but at the price of adding edge-colors.
        Our reductions in what follows share similarities with those of Tranquilli in~\cite{Tra26}.

        \begin{mytheorem}
            Conjecture~\ref{conj:edge-colors} for general graphs is equivalent to the Erd\H{o}s-Gy\'arf\'as conjecture for bipartite graphs.
        \end{mytheorem}
        \begin{proof}
            In one direction, let $(G,\texttt{col})$ be an arbitrary edge-colored graph.
            Assume without loss of generality the set of edge-colors equals $\{1,2,\ldots,k\}$ for some integer $k \ge 1$.
            Let $U = \{u_1,u_2,\ldots,u_k\}$ be made of $k$ fresh new vertices (\textit{i.e.}, not in $V(G)$).
            We construct a bipartite graph $B$ with partite sets $V(G)$ and $U$, such that there is an edge $vu_i$ if and only if $v$ is incident in $G$ to some edge of color $i$.
            By doing so, if every vertex $v$ is incident to at least three colors in $G$, then, it has degree at least three in $B$.
            Similarly, if the edges of any color $i$ induce a clique of size at least three in $G$, then, the degree of $u_i$ in $B$ is also at least three.
            {Under these assumptions, $B$ contains no cycle of length four.
            Indeed, a cycle $x\,u_i\,y\,u_j\,x$ with $i \ne j$ would force the edge $xy$ of $G$ to have both colors $i$ and $j$, which is impossible.}
            Assuming the latter, we further claim that the existence of a rainbow cycle of length $t$ {($t \ge 3$)} in $(G,\texttt{col})$ is equivalent to that of a simple cycle of length $2t$ in $B$.
            Indeed, let $v_0,v_1,\ldots,v_{t-1},v_0$ be any rainbow cycle of $(G,\texttt{col})$.
            For every $i$ with $0 \le i < t$, let $j(i) = \texttt{col}(v_iv_{i+1})$ (indices are taken modulo $t$).
            Then, $v_0,u_{j(0)},v_1,u_{j(1)},\ldots,v_{t-1},u_{j(t-1)},v_0$ is a simple cycle of $B$.
            Conversely, let $v_0,u_{j(0)},v_1,u_{j(1)},\ldots,v_{t-1},u_{j(t-1)},v_0$ be any simple cycle of $B$.
            For every $i$ with $0 \le i < t$, by construction of $B$, the vertices $v_i,v_{i+1}$ are both incident in $G$ to some edge of color $j(i)$.
            As we assume that all the edges of color $j(i)$ induce a clique of $G$, we obtain that $v_iv_{i+1}$ is an edge of $G$ such that $\texttt{col}(v_iv_{i+1}) = j(i)$.
            Hence, $v_0,v_1,\ldots,v_{t-1},v_0$ is a rainbow cycle of $(G,\texttt{col})$.
            {Assuming the Erd\H{o}s-Gy\'arf\'as conjecture for bipartite graphs, $B$ contains a cycle of length $2^q$ with $q \ge 3$, since it has no $C_4$.
            Applying the above transformation with $t=2^{q-1}$, we obtain the existence of a rainbow cycle of power-of-two length in $(G,\texttt{col})$, thus proving Conjecture~\ref{conj:edge-colors}.}

            In the other direction, let $B=(X\cup Y,E)$ be a bipartite graph with respective partite sets $X$ and $Y$.
            Assume further there is no cycle of length four in $B$.
            Let the half-square $G_X$ be the graph with vertex set $X$ such that there is an edge $xx'$ in $G_X$ if and only if $x$ and $x'$ have a common neighbor in $B$.
            We define a function $\texttt{col} : E(G_X) \mapsto Y$ such that, for every edge $xx'$, $\texttt{col}(xx')$ is a vertex $y \in N_B(x) \cap N_B(x')$.
            Note that since we assume there is no cycle of length four in $B$, there is only one possibility for $\texttt{col}(xx')$.
            By construction of $G_X$, each vertex $x$ is incident to an edge of color $y$ for every of its neighbors $y$ in $B$.
            Therefore, if $x$ has degree at least three in $B$, then it is incident to at least three colors in $G_X$.
            Similarly, for every $y \in Y$, the set $N_B(y)$ is a clique of $G_X$ that is induced by all its edges with color $y$.
            So, if $d_B(y) \ge 3$ for every $y \in Y$, then the edges of any given color induce a clique of size at least three.
            To prove that Conjecture~\ref{conj:edge-colors} implies the Erd\H{o}s-Gy\'arf\'as conjecture for bipartite graphs, we prove next that there is a simple cycle of length $2t$ in $B$ ($t \ge 3$) if and only if there is a rainbow cycle of length $t$ in $(G_X,\texttt{col})$. Indeed, if $x_0,y_0,x_1,y_1,\ldots,x_{t-1},y_{t-1},x_0$ is a simple cycle of $B$, then, $x_0,x_1,\ldots,x_{t-1},x_0$ is a cycle of $G_X$ where the edges have the pairwise different colors $y_0,y_1,\ldots,y_{t-1}$. Conversely, let $x_0,x_1,\ldots,x_{t-1},x_0$ be a rainbow cycle of $(G_X,\texttt{col})$, and for every $i$ with $0 \le i < t$, let $y_i := \texttt{col}(x_ix_{i+1})$ (indices are taken modulo $t$). Since the cycle is rainbow, the vertices $y_0,y_1,\ldots,y_{t-1}$ are pairwise different. Furthermore, by construction of $G_X$, $y_i \in N_B(x_i) \cap N_B(x_{i+1})$ for every $i$. Therefore, $x_0,y_0,x_1,y_1,\ldots,x_{t-1},y_{t-1},x_0$ is a simple cycle of $B$.
        \end{proof}

\section{Computer searches}\label{sec:verif}

{We exploit the structural results of Section~\ref{sec:eg-density} to
restrict the search for minimal counter-examples in three graph families.
For general graphs, we can use both the restrictions proved using deletions
in Section~\ref{ssec:setup} and the stronger bounds obtained by the doubling
reduction in Section~\ref{ssec:doubling-reduction}. However, for bipartite graphs, we
only use the former, since edge contractions need not preserve
bipartiteness. For cubic graphs, all these degree restrictions hold
trivially, so we impose degree exactly three and, in some searches,
connectedness.}
We verified the Erd\H{o}s--Gy\'arf\'as conjecture on the three aforementioned families of graphs
by exhaustive search.
The searches use the SAT modulo symmetries framework of Kirchweger and
Szeider~\cite{KiSz21}, which enumerates the isomorphism classes of graphs
satisfying a propositional formula.  We incorporate degree constraints from
Section~\ref{sec:eg-density} into this formula and compare running times with
and without them in Section~\ref{ssec:ablation}.

\begin{mytheorem}\label{thm:verification}
Let $G$ be a graph of order $n$ and minimum degree at least three, such that at least
one of the following conditions holds:
\begin{enumerate}
    \item $n \leq 40$;
    \item $G$ is bipartite and $n \leq 66$;
    \item $G$ is cubic and $n \leq 48$.
\end{enumerate}
Then $G$ contains a cycle whose length is a power of two.
\end{mytheorem}

The previous bound for general graphs was $31$~\cite{Bal26}.  For bipartite
graphs, the best bound covering the whole class was also $31$, because the
bound of $58$ obtained in~\cite{Tra26} is restricted to cubic bipartite
graphs.

    \subsection{The search model}\label{ssec:model}

    Fix an order $n$.  A search state is a graph on the vertex set
    $\{1,\ldots,n\}$, represented by one Boolean variable for each unordered
    pair of distinct vertices, the variable being true if and only if the pair forms an
    edge.  The solver receives a formula in conjunctive normal form over
    these variables together with a list of forbidden subgraphs.  At order
    $n$ the list contains the cycles of length $4,8,\ldots,2^{\lfloor \log_2
    n\rfloor}$, which are the power-of-two lengths realizable on $n$
    vertices.  An occurrence of a forbidden subgraph in a partially assigned
    graph is detected by a propagator, which is built on the Glasgow subgraph
    solver~\cite{McPr20}. The propagator is invoked on every $k$-th call of the underlying
    satisfiability solver, which is CaDiCaL~\cite{cadical}.  We call $k$ the
    check frequency; its default value is $30$.
    States are explored modulo isomorphism.  A second propagator rejects an
    adjacency matrix that is not lexicographically minimal in its isomorphism
    class.  The minimality test has a step limit; branches for which it is
    inconclusive are retained.  A class may therefore be reported more than
    once, but no class is lost.  Every count reported below is zero, and a
    count of zero means that no graph on $n$ vertices satisfies the formula
    while avoiding every forbidden cycle.
    Each order was searched in a single solver run that exhausted the search
    space and reported the number of graphs found.

    \subsection{Encodings}\label{ssec:encodings}

    Each formula is assembled from the constraint groups below, which are all invariant
    under vertex relabeling.  This invariance property is necessary because we have no control on the representative of each isomorphism class, which is selected by  
    the minimality
    propagator.  In particular, by requiring
    a prescribed vertex order, such as placing the degree-three vertices
    first, we could exclude a class whose minimal representative violates that
    order.

    \begin{itemize}
      \item[$(D)$] every vertex has degree at least three;
      \item[$(K)$] every vertex has degree exactly three;
      \item[$(R_1)$] no edge joins two vertices of degree at least four (follows from
            Lemma~\ref{lem:red}(1));
      \item[$(R_2)$] every vertex of degree three has a neighbor of degree
            three (follows from Lemma~\ref{lem:red}(2));
      \item[$(C_1)$] $h \leq \lfloor n/3 \rfloor$ (follows from
            Proposition~\ref{prop:density});
      \item[$(C_2)$] $h \leq \lfloor n/3 \rfloor - 1$ (follows from
            Theorem~\ref{thm:density});
      \item[$(C_3)$] $|L_2| \leq {\max\{0,2h-3\}}$ (follows from Lemma~\ref{lem:D});
      \item[$(B)$] the graph is bipartite;
      \item[$(N)$] the graph is connected.
    \end{itemize}

    Degree conditions and the cardinality bounds $(C_1)$, $(C_2)$ and $(C_3)$
    are expressed with sequential counters: an output variable is true if and
    only if the corresponding sum reaches its specified threshold.
    For the constraint $(C_3)$, we stress that if  $h \le 1$, for which $L_2=\emptyset$, then the upper bound simplifies to $0$.
    We only used  $(C_3)$ in
    encodings that also contain $(R_1)$.
    Then, to deal with $(C_3)$ the program simply counts the vertices having exactly two neighbors
    of degree at least four. Indeed, under $(R_1)$ a vertex of degree at least four
    has no neighbor of degree at least four, and so the counted vertices are exactly the
    vertices of $L_2$ as defined in the proof of
    Proposition~\ref{prop:density}.

    Bipartiteness is expressed by introducing one additional Boolean variable for each
    vertex, with the requirement that
    the two ends of every edge receive different values.  The value at vertex
    $1$ is fixed, which breaks the symmetry between the two possible valuations while
    not excluding any bipartite graph. We stress that these variables are existentially quantified
    witnesses and they take no part in the isomorphism test.

    Six encodings are used:
    \[
      \begin{array}{ll}
        E_0 = (D), &
        E_1 = (D),(R_1),(R_2),(C_1), \\
        E_2 = (D),(R_1),(R_2),(C_2),(C_3), \qquad &
        E_3 = (D),(R_1),(R_2),(C_1),(B), \\
        E_4 = (K), &
        E_5 = (K),(N).
      \end{array}
    \]

    {The constraint of $E_0$ holds for every counter-example, and those of $E_1$ hold for every minimal counter-example within a monotone class of graphs, such as the general graphs and the bipartite graphs.
    By contrast, the encoding $E_2$ is only valid for the search of a minimal counter-example within general graphs.}
    Indeed, this is because it includes the constraints $(C_2)$ and $(C_3)$, respectively based on Theorem~\ref{thm:density} and Lemma~\ref{lem:D}, the proofs of which rely partly on edge contractions.
    As the latter may not preserve bipartiteness, in the encoding $E_3$, which we use for the search of bipartite counter-examples, we include $(C_1)$ but neither $(C_2)$ nor $(C_3)$. 
    In a cubic graph every vertex has degree three, so $H$ is empty and each
    of $(R_1)$, $(R_2)$, $(C_1)$, $(C_2)$ and $(C_3)$ is trivially satisfied. Consequently, to search for cubic counter-examples, we use both encodings $E_4$ and $E_5$ where we impose the degree-three
    condition directly.
    Note that the constraint $(N)$ is sound for a minimal cubic counter-example,
    which is because a disconnected cubic counter-example has a connected component that is itself a
    cubic counter-example of smaller order.

    \subsection{Coverage}\label{ssec:coverage}

    \begin{proof}[Proof of Theorem~\ref{thm:verification}]
    For (1), suppose that there exists a counter-example of order at most $40$.
    In particular, given a minimal counter-example $G$, it holds $n_G \le 40$.
    Recall that for $n \leq
    15$ the conjecture was verified in~\cite{Royle}, and for $n=16,17$
    in~\cite{Bal26}. Therefore, $n_G \geq 18$ holds.  For every order $n$ with $18 \leq
    n \leq 30$, the search with encoding $E_0$ reports that there is no counter-example of order $n$.  Since $E_0$ imposes no
    condition beyond the minimum degree, that search excludes every
    counter-example of order $n$, minimal or not, and therefore $n_G \geq
    31$.  Thus, it holds $31 \leq n_G \leq 40$.  By Lemma~\ref{lem:red}, $G$
    satisfies $(R_1)$ and $(R_2)$; by Proposition~\ref{prop:density} it
    satisfies $(C_1)$; by Theorem~\ref{thm:density} it satisfies $(C_2)$; and
    by Lemma~\ref{lem:D} it satisfies $(C_3)$. It implies that at order $n_G$, the search with either encoding $E_1$ or $E_2$ should detect the representative of the isomorphism class of $G$.
    However, both searches at that order
    report zero, a contradiction.

    For (2), suppose there exists a minimal bipartite counter-example $G$ with $n_G \leq 66$.  The statements of Section~\ref{ssec:setup} hold for a
    minimal counter-example within any {monotone class of graphs}, and bipartite graphs {form such a class}. Therefore, $G$
    satisfies $(R_1)$, $(R_2)$ and $(C_1)$.
    By part (1), $n_G \geq 41$.  However, for every $n$ with $41 \leq n \leq 66$ the
    search with encoding $E_3$ reports zero, which contradicts the existence
    of $G$.

    Finally, for (3), suppose there exists a cubic counter-example of order at most $48$.
    Let $G$ be a minimal cubic counter-example.
    In particular, it holds $n_G \le 48$.
    By part (1), $n_G \geq 41$ holds.
    Since a cubic graph has even order, it implies $n_G \geq 42$.
    For each even $n$ with $42 \leq n \leq 48$ the search with encoding $E_4$
    reports that a cubic counter-example of order $n$ does not exist.  The latter contradicts the existence of the graph $G$.
    \end{proof}

    \begin{myremark}
    Bipartite orders $30$ to $40$ and cubic even orders $26$ to $40$ were
    searched as well, and both searches reported zero.  Those searches confirm part (1)
    independently within the bipartite graphs and the cubic graphs.
    \end{myremark}

    \subsection{Results}\label{ssec:results}

    Tables~\ref{tab:general}, \ref{tab:bipartite} and~\ref{tab:cubic} summarize
    the main verification searches.  Times are measured in wall-clock seconds
    on a single core.  Runs were distributed over four machines; the captions
    identify the groups with comparable running times.

\begin{table}[ht]
\centering
\begin{tabular}{rrrr}
\hline
$n$ & $E_0$ & $E_1$ & $E_2$ \\
\hline
18 & 5.5 & & \\
19 & 21.0 & & \\
20 & 24.7 & & \\
21 & 21.1 & & \\
22 & 132.5 & & \\
23 & 314.1 & & \\
24 & 264.3 & & \\
25 & 695.3 & & \\
26 & 837.1 & & \\
27 & 1810.0 & & \\
28 & 3939.2 & & \\
29 & 5460.6 & & \\
30 & 9388.7 & & \\
31 & & 277.0 & 4829.2$^{\dagger}$ \\
32 & & 511.6 & 12\,188.1$^{\dagger}$ \\
33 & & 1010.2 & 1090.6 \\
34 & & 1937.0 & 2153.2 \\
35 & & 3783.3 & 4248.1 \\
36 & & 7469.3 & 7355.2 \\
37 & & 16\,812.5 & 17\,393.4 \\
38 & & 31\,354.0 & 33\,994.7 \\
39 & & 129\,176.0 & 117\,305.5 \\
40 & & 221\,415.7 & 220\,309.4 \\
\hline
\end{tabular}
\caption{Exhaustive searches for an $n$-vertex graph of minimum degree at least three and no cycle whose length is a power of two. All the entries are reported as
single-core wall-clock times in seconds. No counter-example was found during those searches.
The $E_0$ column, and every entry of $E_1$ and $E_2$ except the two marked
$\dagger$, ran on one machine.  The $E_0$ column used the default check
frequency of $30$. The unmarked entries of $E_1$ and $E_2$ used check
frequency $5$.  The two entries marked $\dagger$ ran on two other machines at
the default check frequency.}
\label{tab:general}
\end{table}

\begin{table}[ht]
\centering
\begin{tabular}{rr@{\qquad}rr@{\qquad}rr}
\hline
$n$ & time & $n$ & time & $n$ & time \\
\hline
30 & 6.5   & 43 & 114.5  & 56 & 4988.0 \\
31 & 6.7   & 44 & 149.5  & 57 & 7757.5 \\
32 & 9.9   & 45 & 175.6  & 58 & 13\,544.4 \\
33 & 12.2  & 46 & 197.0  & 59 & 22\,768.8 \\
34 & 15.1  & 47 & 245.7  & 60 & 60\,214.9 \\
35 & 17.6  & 48 & 409.9  & 61 & 58\,913.0 \\
36 & 20.1  & 49 & 475.9  & 62 & 54\,521.8 \\
37 & 23.8  & 50 & 715.0  & 63 & 104\,835.6 \\
38 & 38.8  & 51 & 898.3  & 64 & 167\,613.0 \\
39 & 38.5  & 52 & 1092.5 & 65 & 254\,008.9 \\
40 & 52.4  & 53 & 1979.1 & 66 & 202\,884.8 \\
41 & 70.0  & 54 & 2611.7 &    &  \\
42 & 80.9  & 55 & 4471.6 &    &  \\
\hline
\end{tabular}
\caption{Exhaustive searches for a bipartite counter-example of order $n$.
All the entries are reported as
single-core wall-clock times in seconds. No counter-example was found during those searches.
All entries use encoding $E_3$ and a
forbidden-subgraph check frequency of $5$.  Orders $64$ to $66$ forbid the
$64$-cycle in addition to the cycles of length $4$, $8$, $16$ and $32$.  Four
machines were used: for the orders $30$ to $54$, $55$ to $60$, $61$ to $64$
and $65$ to $66$ respectively.  Orders $61$ to $64$ ran on a faster machine than orders
$55$ to $60$, which is why the times at orders $61$ and $62$ fall below the
time at order $60$.}
\label{tab:bipartite}
\end{table}

\begin{table}[ht]
\centering
\begin{tabular}{rrr}
\hline
$n$ & $E_4$ & $E_5$ \\
\hline
26 & 11.2 & \\
28 & 33.3 & \\
30 & 90.2 & \\
32 & 298.8 & \\
34 & 678.2 & 180.9 \\
36 & 2157.8 & 575.6 \\
38 & & 1813.4 \\
40 & 10\,970.4 & 6944.4 \\
42 & 25\,211.2 & 20\,723.6 \\
44 & 60\,028.7 & 32\,640.9 \\
46 & 163\,436.1 & \\
48 & 325\,809.2 & \\
\hline
\end{tabular}
\caption{Exhaustive searches for a cubic counter-example of order $n$.  
All the entries are reported as
single-core wall-clock times in seconds. No counter-example was found during those searches.
The encoding $E_4$ solely imposes degree exactly three. For $E_5$, connectedness is further imposed through a built-in option of the solver, and $4$-cycle freeness through additional clauses in the formula. In particular, to forbid the existence of a $C_4$, if we use $E_4$ then we use the propagator, whereas if we use $E_5$ then the condition is directly encoded in the formula.
In
the $E_4$ column, orders $26$ to $36$ used the default check frequency of
$30$ and ran on one machine, orders $40$ to $44$ used frequency $5$ on a
second, and orders $46$ and $48$ used frequency $5$ on a third; the $E_5$
column used frequency $5$ on a fourth.  Order $46$ was searched twice with
$E_4$, on two machines, respectively in $163\,436.1$ and $133\,424.6$ seconds.}
\label{tab:cubic}
\end{table}

    The searches reported in the three tables total about $30$ core-days.

    \subsection{Effect of the structural constraints}\label{ssec:ablation}
    We made some comparative studies to evaluate the runtime impact of integrating the constraints derived from
    Section~\ref{sec:eg-density} in our encoding. More specifically, we compared the searches at orders $22$ and $29$ for both encodings $E_0$ and $E_2$.
    Both runs used the same machine, with one core per run and the same
    forbidden-subgraph check frequency.  At order $22$ the search took
    $132.5$ seconds with $E_0$ and $50.2$ seconds with $E_2$, hence an improvement with $E_2$ by a factor of
    $2.6$.  At order $29$ the search took $5460.6$ seconds with $E_0$ and $1110.5$
    seconds with $E_2$, hence an improvement with $E_2$ by a factor of about $4.9$.  
    Then, we evaluated the impact of adding different chains of constraints in the encoding.
    For that, orders $31$ to $40$ were verified twice: once with $E_1$ and once with
    $E_2$.  The two encodings rely on different results of
    Section~\ref{sec:eg-density}, and each chain of constraints is complete on its own.  For each $n$ with $33 \le n \le 40$, on
    the same machine and at the same check frequency, the relative percentage difference between the running times for both searches was at most $13$ percent.
    Next, we evaluated the impact of check frequency as follows.
    At order $29$ with encoding $E_2$, the search on one machine took
    $1110.5$ seconds at the default check frequency of $30$ and $106.4$
    seconds at frequency $5$.  The captions of
    Tables~\ref{tab:general} to~\ref{tab:cubic} record the frequency used for
    each entry.

    {Finally, we also considered adding the biconnectivity constraint to $E_1$ and $E_2$ at
    orders $28$, $30$ and $32$. Indeed, this restriction is valid for minimal
    counter-examples of even order by Corollary~\ref{cor:even-biconnected}.
    We encoded it in CNF by requiring the existence of a rooted spanning tree in every
    vertex-deleted graph, using auxiliary binary ranks.
    In the completed matched comparisons on the same machine, with
    forbidden-subgraph check frequency of $5$, we observed that adding this constraint increased
    wall-clock time by approximately $41$--$62$ percent.  The solver
    statistics suggest that the cost of the larger encoding outweighed any
    benefit from restricting the search space.  The main verification searches
    do not impose biconnectivity.}

    \subsection{Validation}\label{ssec:validation}
    This subsection describes the tests that we ran to check that our
    implementation was correct. Indeed, an implementation error could make
    a search return zero graphs. First, to check that no constraint excludes
    graphs that it should keep, we verified that each formula accepts a graph
    (that is, remains satisfiable once the edges of the graph are fixed)
    exactly when the graph satisfies the corresponding constraints: for $E_2$
    on all $236\,926$ labeled graphs of order $7$ with minimum degree at least
    three, and for $E_1$, $E_2$ and $E_4$ on the graphs of orders $10$ to $12$
    generated by \texttt{geng}~\cite{nauty}.  Second, to check that the
    pipeline does report graphs when some exist, we ran it on two instances with
    a nonzero answer: with only the $4$-cycle forbidden, it finds $5$ graphs
    of order $10$ and $36$ cubic graphs of order $14$ up to isomorphism, the
    same counts as \texttt{nauty}.

    \subsection{Reproducibility}\label{ssec:repro}

    The repository accompanying this paper\footnote{{\href{https://github.com/bogdan27182/eg-paper}{GitHub repository}}}
    contains the programs that build the encodings, the validation suites presented in
    Section~\ref{ssec:validation}, the pinned versions of the three solvers,
    and, for every search, its full output and a record of its
    result.
    The solver was not asked to produce a resolution proof. Therefore, a search must be rerun in order to verify the result. 

\section{Perspectives}\label{sec:perspectives}

There is room for improving our understanding of the structure of minimal counter-examples (with or without additional assumptions such as being cubic or bipartite).
For instance, finding better upper bounds on their density could help in speeding up the computer searches.
Since most vertices in any minimal counter-example have degree three, we ask whether the number of edges is at most $\frac 3 2n + o(n)$.
We also leave open the related question of whether the number of vertices of degree at least four is at most $n/k$, for some absolute constant $k > 3$.
The result of Theorem~\ref{thm:Dplus} could be a good start to deal with this question.
On a different note, we also conjecture that any minimal counter-example (respectively any minimal cubic counter-example) is biconnected, and that any minimal bipartite counter-example is $2$-edge connected.
Finally, we observe that any counter-example is $K_4$-free and $K_{2,2}$-free, which implies bounded treewidth for certain classes of graphs~\cite{Dal21}.
The latter could help in proving the conjecture for new classes of graphs.

\bibliographystyle{amsplain}
\bibliography{biblio}

\end{document}